\documentclass[11pt]{amsart}
\usepackage[dvipsnames]{xcolor}
\usepackage{amssymb,amsmath,amsthm,enumerate,mathtools,mathptmx}
\usepackage[new]{old-arrows}
\usepackage{tikz-cd} 
\usepackage[utf8]{inputenc}

\usepackage{hyperref}
\hypersetup{%
	colorlinks = true,
	linkcolor = BrickRed,
	citecolor = Green,
	urlcolor   = blue,
  filecolor = red,
}

\usepackage{cleveref}
\usepackage{enumitem}
\usepackage[margin=0.9in]{geometry}
\usepackage{parskip}

\usepackage[backend=biber,style=alphabetic,doi=false,isbn=false,url=false,eprint=false,maxbibnames=5,minbibnames=5,mincitenames=5,maxcitenames=5,maxalphanames=5,minalphanames=5,backref=true]{biblatex}
\DeclareFieldFormat{extraalpha}{#1}
\DeclareLabelalphaTemplate{%
  \labelelement{%
    \field[final]{shorthand}
    \field{label}
    \field[strwidth=2,strside=left,ifnames=1]{labelname}
    \field[strwidth=1,strside=left]{labelname}
  }
}
\DefineBibliographyStrings{english}{%
  backrefpage={},
  backrefpages={}
}

\usepackage{xpatch}
\DeclareFieldFormat{backrefparens}{\addperiod#1}
\xpatchbibmacro{pageref}{parens}{backrefparens}{}{}

\newtheorem{thm}{Theorem}[section]
\newtheorem{lem}[thm]{Lemma}
\newtheorem{prop}[thm]{Proposition}

\newtheorem{conj}[thm]{Conjecture}
\newtheorem{identity}[thm]{Identity}

\theoremstyle{definition}
\newtheorem{rem}[thm]{Remark}
\newtheorem{defn}[thm]{Definition}
\newtheorem{example}[thm]{Example}

\numberwithin{equation}{section}

\crefname{thm}{theorem}{theorems}
\crefname{rem}{remark}{remarks}
\crefname{prop}{proposition}{propositions}
\crefname{lem}{lemma}{lemmas}
\crefname{identity}{identity}{identities}
\crefname{equation}{}{}

\DeclareMathOperator{\GL}{GL}
\DeclareMathOperator{\SL}{SL}
\DeclareMathOperator{\SU}{SU}
\DeclareMathOperator{\UU}{U}
\DeclareMathOperator{\OO}{O}
\DeclareMathOperator{\GG}{G}
\DeclareMathOperator{\Sp}{Sp}
\DeclareMathOperator{\SpU}{SpU}
\DeclareMathOperator{\diag}{diag}
\DeclareMathOperator{\Sym}{Sym}
\DeclareMathOperator{\Pf}{Pf}
\DeclareMathOperator{\chr}{char}
\DeclareMathOperator{\adj}{adj}

\newcommand{\md}[1]{{\left\lvert #1 \right\lvert}}
\newcommand{\deff}[1]{{\color{Cerulean}#1}}
\newcommand{\into}{\longhookrightarrow}
\DeclareRobustCommand{\onto}{\relbar\joinrel\twoheadrightarrow}
\newcommand{\tr}{\operatorname{tr}}
\newcommand{\dG}{{\mathrm{d}}G}
\newcommand{\Sage}{\texttt{SageMath}}

\let\subset\subseteq

\let\ge\geqslant
\let\le\leqslant
\let\mapsto\longmapsto
\let\to\longrightarrow

\begin{document}

\title[Splitting the difference]{Splitting the difference: Computations of the Reynolds operator \\ in classical invariant theory}

\author{Aryaman Maithani}
\address{Department of Mathematics, University of Utah, 155 South 1400 East, Salt Lake City, UT~84112, USA}
\email{maithani@math.utah.edu}

\thanks{The author was supported by NSF grants DMS 2101671 and DMS 2349623.}



\begin{abstract}
	If $G$ is a linearly reductive group acting rationally on a polynomial ring $S$, 
	then the inclusion $S^{G} \into S$ possesses a unique $G$-equivariant splitting, called the Reynolds operator.
	We describe algorithms for computing the Reynolds operator for the \emph{classical actions} as in Weyl's book. 
	The groups are the general linear group, the special linear group, the orthogonal group, and the symplectic group, 
	with their classical representations: direct sums of copies of the standard representation and copies of the dual representation.
\end{abstract}

\maketitle

{\setlength{\parskip}{0em}
\tableofcontents}

\section{Introduction} \label{sec:introduction}

	Consider a group $G$ acting on a ring $S$ by ring automorphisms. 
	The \deff{ring of invariants} for this group action is defined as
	\begin{equation*} 
		S^{G} \coloneqq \{s \in S : g(s) = g \ \text{for all} \ g \in G\},
	\end{equation*}
	that is, $S^{G}$ is the subring of elements that are fixed by each group element. 
	We have the inclusion of rings
	\begin{equation} \label{eq:inclusion}
		S^{G} \into S.
	\end{equation}
	The above is also then an inclusion of $S^{G}$-modules. 
	A natural question to ask is whether~\Cref{eq:inclusion} splits in the category of $S^{G}$-modules --- in which case $S^{G}$ is a direct summand of $S$.
	A positive answer to this question often implies good properties about the subring; for example, a direct summand of a noetherian ring is again noetherian. 
	A deeper result is the Hochster--Roberts theorem~\Cite{HochsterRoberts}, which states that a direct summand of a polynomial ring is Cohen--Macaulay. 
	The inclusion~\Cref{eq:inclusion} does not always split; 
	a simple example is the alternating group $A_{3}$ acting on $\mathbb{F}_{3}[x, y, z]$ by permuting the variables, see \Cref{ex:A3-not-split}. 
	A more dramatic example was given by \Citeauthor{Nagarajan}~\Cite{Nagarajan}, where a group of order two acts on a regular ring for which the ring of invariants is not noetherian. 
	For finite groups, a simple condition that ensures the existence of a splitting is having order invertible in $S$; the inclusion~\Cref{eq:inclusion} then splits with an $S^{G}$-linear splitting given by
	\begin{equation*} 
		s \mapsto \frac{1}{\md{G}} \sum_{g \in S} g(s).
	\end{equation*}
	The above is the \emph{Reynolds operator} and has the additional property of being \emph{$G$-equivariant} (\Cref{defn:splitting}).

	In this paper, our groups of interest are certain linear algebraic groups over a field $k$, that is,
	Zariski-closed subgroups of $\GL_{n}(k)$. 
	If such a group $G$ acts (rationally) on a $k$-vector space $V$, then we get a (rational) degree-preserving $k$-algebra action of $G$ on the polynomial ring $S \coloneqq \Sym(V)$. 
	Hilbert's fourteenth problem asked if $S^{G}$ is always a finitely generated $k$-algebra --- a question answered in the negative by \Citeauthor{Nagata14th}~\Cite{Nagata14th} by giving an example where $S^{G}$ is not noetherian. 
	For linear algebraic groups, the analogue to having invertible order is to be \emph{linearly reductive}. 
	These groups admit a similar Reynolds operator, see \Cref{thm:linearly-reductive-reynolds-unique-linear};
	in particular, the inclusion~\Cref{eq:inclusion} splits $G$-equivariantly and $S^{G}$-linearly. 

	We focus on the following titular \emph{classical groups} of Weyl's book~\Cite{WeylClassical}: 
	the general linear group $\GL_{n}(k)$, 
	the special linear group $\SL_{n}(k)$, 
	the orthogonal group $\OO_{n}(k)$, and
	the symplectic group $\Sp_{2n}(k)$. 
	As in the book, we look at their classical actions, corresponding to the direct sum of copies of the standard representation and possibly copies of the dual representation. 
	We record the rings of invariants for some of these actions in \Cref{thm:classical-invariants}. 
	This includes infinite fields of positive characteristic as in~\Cite{ConciniProcesiCharacteristicFree, Hashimoto:AnotherProof}. 
	There is, however, a stark difference between characteristics zero and positive: 
	if $k$ is a field of characteristic zero, then the groups listed above are all linearly reductive. 
	This is typically not the case in positive characteristic wherein these groups admit representations for which the ring of invariants is not Cohen--Macaulay~\Cite{Kohls:NonCM}. 
	Moreover --- while the classical rings of invariants continue to be Cohen--Macaulay even in positive characteristic --- the inclusion~\Cref{eq:inclusion} is rarely split~\Cite{HochsterJeffriesPandeySingh}. 
	This has the interesting consequence that given any splitting over $\mathbb{Q}$, every prime must appear in the denominator of the image of any basis; see \Cref{rem:primes-in-denominators} for a precise statement. 

	For the most part, we consider these classical groups in characteristic zero. 
	Because these are then linearly reductive, the inclusion~\Cref{eq:inclusion} splits.
	We give an algorithm for explicitly computing the Reynolds operator in each case in terms of certain integrals of monomial functions. 
	We do this by reducing the computation to one over a compact Lie group, in which case we may integrate with respect to the Haar measure akin to averaging over a finite group. 
	Methods to compute these integrals are of interest in mathematical physics due to their important role in areas such as mesoscopic transport, quantum chaos, and quantum information and decoherence. 
	This interest has led to the development of various algorithms --- such as the \emph{invariant method} and the \emph{column vector method} --- to compute these integrals;
	see the introduction of~\Cite{GorinLopez} for more on this topic.

	We remark that there are conditions weaker than having invertible order or being linearly reductive that imply finite generation of $S^{G}$. 
	Indeed, Noether~\Cite{Noether:Invariants} showed that if $G$ is a finite group acting on a finitely generated $k$-algebra $S$ by $k$-algebra automorphisms, then $S^{G}$ is a finitely generated $k$-algebra. 
	Similarly, \Citeauthor{Haboush:Reductive}~\Cite{Haboush:Reductive} proved that if $G$ is a \emph{reductive group} acting rationally on a finitely generated $k$-algebra $S$, then $S^{G}$ is finitely generated. 
	While the classical groups are no longer linearly reductive in positive characteristic, they continue to be reductive, and hence the invariant subrings are known to be finitely generated. 

	The paper is arranged as follows. 
	After setting up the notations and definitions in \Cref{sec:basic-notions}, we define the classical group actions in \Cref{sec:classical-group-actions} and record the rings of invariants. 
	In \Cref{sec:linearly-reductive}, we recall the relevant facts about linearly reductive groups. 
	\Cref{sec:splitting-over-lie-group} discusses the computation of the Reynolds operator for a compact Lie group. We discuss facts about the Haar measure and set up the required machinery to integrate functions that take values in polynomial rings. 
	\Cref{sec:reynolds-classical} begins by describing how the computation of the Reynolds operator for a classical group over an arbitrary field of characteristic zero can be reduced to that for a compact Lie group. 
	With this reduction in place, we then give algorithms that one may implement on a computer algebra system. 
	We make use of these algorithms in \Cref{sec:explicit-formulae} to provide explicit formulae for the Reynolds operators for the $\SL$ and $\GL$ actions. 
	These algorithms have been implemented in \Sage~\Cite{sagemath}, and we note some conjectures arising out of these computations.
	Lastly, we compare with the situation in positive characteristic in \Cref{sec:positive-characteristic}.

\section{Notations and definitions} \label{sec:basic-notions}

	The letter $k$ will denote a field. 
	For $n \ge 1$, $\mathbb{A}_{k}^{n}$ denotes the topological space $k^{n}$ with the Zariski topology. 
	We recall the following classical groups of invertible matrices. 
	\begin{enumerate}[label=(\alph*)]
		\item (General linear group) $\GL_{n}(k)$ is the group of $n \times n$ invertible matrices over $k$.
		\item (Special linear group) $\SL_{n}(k) \coloneqq \{M \in \GL_{n}(k) : \det(M) = 1\}$.
		\item (Orthogonal group) $\OO_{n}(k) \coloneqq \{M \in \GL_{n}(k) : M^{\tr} M = I_{n}\}$, where $I_{n}$ denotes the identity matrix.
		\item (Symplectic group) $\Sp_{2n}(k) \coloneqq \{M \in \GL_{2n}(k) : M^{\tr} \Omega M = \Omega\}$, 
		where $\Omega \coloneqq \left(
		\begin{smallmatrix}
		O & I_{n} \\ 
		-I_{n} & O \\
		\end{smallmatrix}
		\right)$.
	\end{enumerate}

	When the field $k$ is taken to be the complex numbers, we have the following additional subgroups.
	\begin{enumerate}[label=(\alph*), resume]
		\item (Unitary group) $\UU_{n}(\mathbb{C}) \coloneqq \{U \in \GL_{n}(\mathbb{C}) : U U^{\ast} = I_{n}\}$, 
		where $U^{\ast}$ denotes the conjugate transpose of $U$.
		\item (Special unitary group) $\SU_{n}(\mathbb{C}) \coloneqq \UU_{n}(\mathbb{C}) \cap \SL_{n}(\mathbb{C})$. 
		\item (Symplectic unitary group) $\SpU_{2n}(\mathbb{C}) \coloneqq \UU_{2n}(\mathbb{C}) \cap \Sp_{2n}(\mathbb{C})$.
	\end{enumerate}
 
	All the above groups inherit the subspace topology from $\mathbb{A}_{k}^{n^{2}}$, 
	and we refer to this as the Zariski topology.
	These are all topological groups --- though typically not Hausdorff --- because the product and inversion functions are continuous in the Zariski topology,
	being given by rational functions in the entries of the matrices. 

	When $k = \mathbb{C}$, these groups also have the Euclidean topology 
	and moreover are smooth submanifolds of $\mathbb{C}^{n^{2}}$. 
	In this case, the product and inversion functions are smooth; hence, these are all Lie groups. 

	\begin{defn} \label{defn:splitting}
		Let $G$ be a group acting by ring automorphisms on a ring $S$. 
		A \deff{splitting} for the inclusion 
		$S^{G} \into S$ is an additive function $\mathcal{R} \colon S \to S^{G}$ such that 
		$\mathcal{R}(r) = r$ for all $r \in S^{G}$. 
		The splitting is \deff{$G$-equivariant} if $\mathcal{R}(g(s)) = \mathcal{R}(s)$ for all 
		$g \in G$ and $s \in S$. 
		The splitting is \deff{$S^{G}$-linear} if 
		$\mathcal{R}(rs) = r \mathcal{R}(s)$ for all $r \in S^{G}$ and $s \in S$.
	\end{defn}

\section{The classical group actions} \label{sec:classical-group-actions}
	
	Let $k$ be a field, and $t$, $m$, $n$ be positive integers. 
	We use the notation
	\begin{equation*} 
		k[Y_{t \times n}] \coloneqq 
		k[y_{ij} : 1 \le i \le t,\, 1 \le j \le n],
	\end{equation*} 
	that is, $k[Y_{t \times n}]$ is a polynomial ring over $k$ in $tn$ variables. 
	Once the dimensions have been specified, we write $k[Y]$ for brevity.
	We use the letter $Y$ for the $t \times n$ matrix $[y_{ij}]_{i, j}$. 
	The notation naturally extends to $k[X_{m \times t}, Y_{t \times n}]$.

	Let $G$ be one of the groups $\GL_{t}(k)$, $\SL_{t}(k)$, $\OO_{t}(k)$, or $\Sp_{t}(k)$,
	where for the last case, we assume that $t$ is even. 
	We will consider the following two types of rational actions of $G$.
	\begin{enumerate}[label=(R\arabic*)]
		\item \label{item:standard-action} 
		The group $G$ acts on $k[Y_{t \times n}]$, 
		where the action of $M \in G$ is given by
		\begin{equation*} 
			M \colon Y \mapsto M Y;
		\end{equation*}
		by the above, we mean that $[Y]_{ij} \mapsto [MY]_{ij}$.
		\item \label{item:standard-dual-action} 
		The group $G$ acts on $k[X_{m \times t}, Y_{t \times n}]$, 
		where the action of $M \in G$ is given by
		\begin{equation*} 
			M \colon 
			\begin{cases}
				X \mapsto X M^{-1}, \\
				Y \mapsto M Y.
			\end{cases}
		\end{equation*}
	\end{enumerate}
	The first action corresponds to the direct sum of $n$ copies of the standard representation, 
	whereas the second has an additional $m$ copies of the dual representation. 
	We will describe the splittings for all of these actions. 

	We recall below the \emph{classical rings of invariants} as in Weyl's book~\Cite{WeylClassical} where they were originally discussed in characteristic zero. 
	A characteristic-free proof of the following theorem can be found in~\Cite{ConciniProcesiCharacteristicFree, Hashimoto:AnotherProof}.
	\begin{thm} \label{thm:classical-invariants}
		Let $k$ be an infinite field. With the above actions, we have the following rings of invariants.
		\begin{enumerate}[label=(\alph*)]
			\item (General linear group) For positive integers $t$, $m$, $n$, the equality 
			\begin{equation*} 
				k[X_{m \times t}, Y_{t \times n}]^{\GL_{t}(k)}
				=
				k[XY]
			\end{equation*} 
			holds, that is,
			the invariant ring is generated,
			as a $k$-algebra,
			by the entries of the matrix product $XY$.
			\item (Special linear group) For positive integers $t$, $n$ with $t \le n$, the equality 
			\begin{equation*} 
				k[Y_{t \times n}]^{\SL_{t}(k)} = k[\text{size $t$ minors}]
			\end{equation*}
			holds, that is,
			the invariant ring is generated,
			as a $k$-algebra,
			by the size $t$ minors of the matrix $Y$.
			\item (Orthogonal group) For positive integers $t$, $n$ and $\chr(k) \neq 2$, the equality 
			\begin{equation*} 
				k[Y_{t \times n}]^{\OO_{t}(k)} = k[Y^{\tr} Y]
			\end{equation*}
			holds, that is,
			the invariant ring is generated,
			as a $k$-algebra,
			by the entries of the matrix product $Y^{\tr} Y$.
			\item (Symplectic group) For positive integers $t$, $n$, the equality 
			\begin{equation*} 
				k[Y_{2t \times n}]^{\Sp_{2t}(k)} = k[Y^{\tr} \Omega Y]
			\end{equation*}
			holds, that is,
			the invariant ring is generated,
			as a $k$-algebra,
			by the entries of the matrix product $Y^{\tr} \Omega Y$.
		\end{enumerate}
	\end{thm}

	\begin{rem}
		For each of the above actions, the fixed subring is of independent interest for the reasons described below. 
		We denote the invariant subring in the respective cases by $R$.
		\begin{enumerate}[label=(\alph*)]
			\item (General linear group) The ring $R$ is isomorphic to the determinantal ring $k[Z_{m \times n}]/I_{t + 1}(Z)$, 
			where $I_{t + 1}(Z)$ is the ideal generated by the size $t + 1$ minors of $Z$.
			\item (Special linear group) The ring $R$ is the Pl\"ucker coordinate ring of the Grassmannian of $t$-dimensional subspaces of an $n$-dimensional space.
			\item (Orthogonal group) The ring $R$ is isomorphic to $k[Z]/I_{t + 1}(Z)$, where $Z$ is an $n \times n$ symmetric matrix of indeterminates.
			\item (Symplectic group) The ring $R$ is isomorphic to $k[Z]/\Pf_{2t + 2}(Z)$, where $Z$ is an $n \times n$ alternating matrix of indeterminates, and $\Pf_{2t + 2}(Z)$ the ideal generated by its principal $(2t + 2)$-Pfaffians.
		\end{enumerate}
	\end{rem}

\section{Linearly reductive groups} \label{sec:linearly-reductive}

	This section contextualises our results with the broader theory of linearly reductive groups. 
	For the most part, this is only for theoretical interest, as we will compute the Reynolds operator concretely by integrating over a compact Lie group. 
	For an introduction to linear algebraic groups and rational actions, 
	we refer the reader to one of~\Cite{FogartyInvariant, MumfordFourteenthProblem, HochsterInvariantSurvey, DerksenKemper}. 
	We record the relevant facts here. 

	\begin{defn} \label{defn:reynolds-operator}
		Let $G$ be a linear algebraic group over the field $k$, 
		and $V$ a rational representation of $G$. 
		A \deff{Reynolds operator} is a $k$-linear, $G$-equivariant splitting $\mathcal{R} \colon k[V] \to k[V]^{G}$.
	\end{defn}

	\begin{thm} \label{thm:linearly-reductive-reynolds-unique-linear}
		If $G$ is linearly reductive, 
		then for every rational representation $V$, 
		there exists a \emph{unique} Reynolds operator $\mathcal{R} \colon k[V] \to k[V]^{G}$. 
		Moreover, $\mathcal{R}$ is $k[V]^{G}$-linear.
	\end{thm}
	\begin{proof} 
		The statements are Theorem 2.2.5 and Corollary 2.2.7 in~\Cite{DerksenKemper}, respectively.
	\end{proof}

	\begin{example}
		We give an example of a group $G$ acting on a polynomial ring $S$ for which there exists an $S^{G}$\nobreakdash-linear splitting but no $G$-equivariant splitting. 
		Let $G$ be the symmetric group on two element, and $S \coloneqq \mathbb{F}_{2}[x, y]$. 
		The group $G$ acts on $S$ by permuting the variables, and the invariant subring is $\mathbb{F}_{2}[x+y, xy]$. 
		Because $S$ is a free $S^{G}$-module with $\{1, x\}$ as a basis, the inclusion $S^{G} \into S$ splits $S^{G}$-linearly. 
		Suppose that $\pi \colon S \to S^{G}$ is a $G$-equivariant splitting. 
		Then, $\pi(x) = \pi(y)$ because $x$ and $y$ are in the same orbit. But then,
		\begin{equation*} 
			x + y = \pi(x + y) = \pi(x) + \pi(y) = 2 \pi(x) = 0,
		\end{equation*}
		a contradiction. 
		Thus, $S^{G} \into S$ admits no $G$-equivariant splitting even though it splits $S^{G}$-linearly. 
		This example extends mutatis mutandis to any positive characteristic $p$ by considering the permutation action of~$\Sigma_{p}$~--- the symmetric group on $p$ elements --- on the polynomial ring $\mathbb{F}_{p}[x_{1}, \ldots, x_{p}]$.
	\end{example}

	\begin{example} \label{ex:A3-not-split}
		We now give an example of a group action for which no $S^{G}$-linear splitting exists. 
		Consider the action of the alternating group $G \coloneqq A_{3}$ on the polynomial ring $S \coloneqq \mathbb{F}_{3}[x, y, z]$ by permuting the variables. 
		If we let $e_{1}$, $e_{2}$, $e_{3}$ denote the elementary symmetric polynomials in $x$, $y$, $z$ and set $\Delta \coloneqq (x - y)(y - z)(z - x)$, 
		then one can check that 
		$\Delta \in S^{G}$, 
		$\Delta \notin (e_{1}, e_{2}, e_{3}) S^{G}$, but 
		$\Delta \in (e_{1}, e_{2}, e_{3}) S$. 
		This implies that $S^{G} \into S$ does not split over $S^{G}$. 
		More generally, if $A_{n}$ acts on $S = \mathbb{F}_{p}[x_{1}, \ldots, x_{n}]$ by permuting variables, the inclusion $S^{A_{n}} \into S$ splits if and only if $p$ does not divide $\md{A_{n}}$;
		the nontrivial implication was proven 
		in~\Cite[Theorem 12.2]{Glassbrenner:CMFrational} for $p \nmid n(n - 1)$, 
		and the general case can be found in 
		\Cite[Theorem 5.5]{Singh:FailureF}, 
		\Cite{Smith:AlternatingInvariants},
		\Cite[Theorem 2.18]{Jeffries:Thesis}, and 
		\Cite[Corollary 4.2]{GoelJeffriesSingh}.
	\end{example}

	\begin{example}
		If $k$ is a field of characteristic zero, 
		then the classical groups $\GL_{n}(k)$, $\SL_{n}(k)$, $\OO_{n}(k)$, and $\Sp_{2n}(k)$ are all linearly reductive, 
		as are all finite groups.
		For a finite group $G$, 
		the Reynolds operator is just averaging over the group: 
		$\mathcal{R}(f) = \frac{1}{\md{G}} \sum\limits_{g \in G} g(f)$. 
	\end{example}

	The above Reynolds operator extends naturally to smooth actions of a compact Lie group; see \Cref{thm:reynolds-for-lie-group}. 
	The following theorem, 
	in conjunction with \Cref{prop:invariants-and-operator-over-GC-and-intersection}, 
	tells us how the computation of the Reynolds operator for a linearly reductive group over $\mathbb{C}$ can be reduced to that for a compact Lie group.

	\begin{thm} \label{thm:equivalent-linearly-reductive-over-C}
		Let $G$ be a linear algebraic group over $\mathbb{C}$. The following are equivalent.
		\begin{enumerate}[label=(\alph*)]
			\item $G$ is linearly reductive.
			\item $G$ has a Zariski-dense subgroup that is a compact Lie group (in the Euclidean topology).
		\end{enumerate}
	\end{thm} 
	We shall deduce the above theorem for the classical groups of interest by producing Zariski-dense subgroups in \Cref{thm:density}.

\section{The Reynolds operator for a Lie group} \label{sec:splitting-over-lie-group}

	We will now describe the Reynolds operator for a compact Lie group acting on a 
	polynomial ring. 
	Strictly speaking, the term ``Reynolds operator'' was defined for the rational action of a linear algebraic group, but we continue to use this term to mean a ($\mathbb{C}$-)linear $G$-equivariant splitting.
	We first recall some theory of integration over such a group.

	In this section, a finite-dimensional vector space over $\mathbb{R}$ will have its canonical structure of a real differentiable manifold. 
	Examples include $\mathbb{C}$ and finite-dimensional vector spaces over $\mathbb{C}$.
	Let $G$ be a compact real Lie group and $\dG$ denote the (normalised) Haar measure on $G$. 
	Given an element $g \in G$, 
	we denote by $L_{g}$ and $R_{g}$ the left and right translation maps:
	\begin{equation} \label{eq:translation-maps}
		\begin{aligned}
			L_{g} \colon G &\to G, \\
						 h &\mapsto gh,
		\end{aligned}
		\qquad\qquad
		\begin{aligned}
			R_{g} \colon G &\to G, \\
						 h &\mapsto hg.
		\end{aligned}
	\end{equation}

	For an introduction to the Haar measure, we refer the reader to one of~\Cite{HalmosMeasure, RoydenAnalysis, LangAnalysis}. 
	We next recall the properties of interest to us.

	\begin{thm} \label{thm:invariance-to-field}
		Let $\psi \colon G \to \mathbb{R}$ be smooth, and $g \in G$. Then,
		\begin{equation*} 
			\int_{G} \psi \,\dG 
			= \int_{G} (\psi \circ L_{g}) \,\dG 
			= \int_{G} (\psi \circ R_{g}) \,\dG.
		\end{equation*}
		If $\psi$ is constant and takes the value $1$, then
		\begin{equation*} 
			\int_{G} \psi \, \dG = 1.
		\end{equation*}
	\end{thm}

	We may naturally extend the integration of scalar-valued functions to vector-valued functions:
	\begin{defn}
		Let $V$ be a finite-dimensional $\mathbb{R}$-vector space, 
		and $\psi \colon G \to V$ a smooth function. 
		Fix a basis $\{v_{1}, \ldots, v_{n}\}$ of $V$. 
		Let $\psi_{i} \colon G \to \mathbb{R}$ be the corresponding coordinate functions, 
		satisfying $\psi(g) = \sum \psi_{i}(g) v_{i}$. 
		We define
		\begin{equation*} 
			\int_{G} \psi \coloneqq \sum_{i = 1}^{n} \left(\int_{G} \psi_{i} \,\dG\right) v_{i} \in V.
		\end{equation*}
	\end{defn}
	One checks that the above definition is independent of the choice of basis. 
	Note that our notation above drops the ``$\dG$'' when integrating vector-valued functions. 
	This is for ease of notation as we will always be integrating with respect to the Haar measure.
	The linearity of scalar integration and the properties of the Haar measure readily extend to the following.
	\begin{lem} \label{lem:integral-commute-linear-maps}
		Let $T \colon V \to W$ be a linear map of finite-dimensional vector spaces, and let $\psi \colon G \to V$ be a smooth function. Then,
		\begin{equation*} 
			\int_{G} (T \circ \psi) = T\left(\int_{G} \psi\right).
		\end{equation*}
	\end{lem}

	\begin{lem} \label{lem:invariance-to-vector-space}
		Let $\psi \colon G \to V$ be smooth, and $g \in G$. Then,
		\begin{equation*} 
			\int_{G} \psi = \int_{G} (\psi \circ L_{g}) = \int_{G} (\psi \circ R_{g}).
		\end{equation*}
		If $\psi$ is constant and takes the value $v$, then
		\begin{equation*} 
			\int_{G} \psi = v.
		\end{equation*}
	\end{lem}

	\begin{defn}
		Suppose $V$ is an infinite-dimensional vector space, 
		and $\Psi \colon G \to V$ a function such that 
		the vector space spanned by the image of $\Psi$ is finite-dimensional. 
		Let $W \subset V$ be any finite-dimensional subspace containing the image of $\Psi$, 
		and let $\psi \colon G \to W$ be the restriction of $\Psi$. 
		We say that $\Psi$ is \deff{smooth} if $\psi$ is smooth, and define
		\begin{equation*} 
			\int_{G} \Psi \coloneqq \int_{G} \psi,
		\end{equation*}
	\end{defn}
	where we note that the above definitions are independent of the choice of $W$.

	Let $S = \mathbb{C}[x_{1}, \ldots, x_{n}]$ be a polynomial ring, 
	and let $[S]_{1}$ denote the $\mathbb{C}$-vector space of homogeneous degree one polynomials. 
	There is a natural isomorphism of groups
	\begin{equation*} 
		\{\text{degree-preserving $\mathbb{C}$-algebra automorphisms of $S$}\} \longleftrightarrow \{\text{$\mathbb{C}$-linear automorphisms of $[S]_{1}$}\}.
	\end{equation*}

	A degree-preserving $\mathbb{C}$-algebra action of $G$ on $S$ is called \deff{smooth} if the corresponding action
	$G \times [S]_{1} \to [S]_{1}$ is smooth. 
	In this case, the corresponding action $G \times [S]_{d} \to [S]_{d}$ is smooth for all $d \ge 0$, 
	where $[S]_{d}$ denotes the space of homogeneous polynomials of degree $d$.
	For $f \in S$, define the orbit map
	\begin{align*} 
		\psi_{f} \colon G &\to S \\
		g &\mapsto g(f).
	\end{align*}
	The function $\psi_{f}$ takes values within a finite-dimensional subspace of $S$, 
	for example, 
	the space of polynomials of degree at most the degree of $f$. 
	If the $G$-action is smooth, then $\psi_{f}$ is a smooth function.

	\begin{thm} \label{thm:reynolds-for-lie-group}
		Let $G$ be a compact Lie group acting smoothly on the polynomial ring $S \coloneqq \mathbb{C}[x_{1}, \ldots, x_{n}]$ by degree-preserving $\mathbb{C}$\nobreakdash-algebra automorphisms. Then, $S^{G} \into S$ splits with a degree-preserving, $G$-equivariant, $S^{G}$-linear splitting $\mathcal{R} \colon S \onto S^{G}$ given by
		\begin{equation*} 
			\mathcal{R} \colon f \mapsto \int_{G} \psi_{f}.
		\end{equation*}
		Suggestively, the above may be written as
		\begin{equation*} 
			\mathcal{R}(f) = \int_{g \in G} g(f), 
		\end{equation*}
		resembling the Reynolds operator for finite groups.
	\end{thm}
	\begin{proof} 
		The $\mathbb{C}$-linearity of $\mathcal{R}$ is clear. 
		If $f$ is homogeneous, then $\psi_{f}$ takes values in 
		subspace $[S]_{\deg(f)}$ and in turn, 
		$\mathcal{R}(f) \in [S]_{\deg(f)}$. 
		Thus, $\mathcal{R}$ is a degree-preserving $\mathbb{C}$-linear map. 

		For the rest of the proof, 
		we will make repeated use of \Cref{lem:integral-commute-linear-maps,lem:invariance-to-vector-space}. 
		Recall that $L_{g}$ and $R_{g}$ denote the translation maps, 
		defined in~\Cref{eq:translation-maps}.
		For $f \in S$ and $g \in G$, we define the $\mathbb{C}$-linear functions
		$S \xrightarrow{\rho_{f}} S$ and $S \xrightarrow{\mu_{g}} S$ 
		given by left multiplication and the $G$-action, respectively. 
		Consequently, 
		\begin{align*} 
			\mathcal{R}(f)
			&= \int_{G} \psi_{f} 
			= \int_{G} \psi_{f} \circ R_{g}  
			= \int_{G} \psi_{g(f)} 
			= \mathcal{R}(g(f)) \\[5pt]
			&= \int_{G} \psi_{f} \circ L_{g}
			= \int_{G} \mu_{g} \circ \psi_{f}
			= \mu_{g}\left(\int_{G} \psi_{f}\right)
			= g(\mathcal{R}(f)).
		\end{align*}
		The above shows that $\mathcal{R}$ takes values in $S^{G}$ and is $G$-equivariant. 
		Lastly, if $f \in S^{G}$ and $h \in S$, then
		\begin{equation*} 
			\mathcal{R}(fh) = \int_{G} \psi_{fh} = \int_{G} \rho_{f} \circ \psi_{h} = \rho_{f} \left(\int_{G} \psi_{h}\right) = f \mathcal{R}(h),
		\end{equation*}
		and $\psi_{f}$ is identically equal to $f$, giving us
		\begin{equation*} 
			\mathcal{R}(f) = \int_{G} \psi_{f} = f.
		\end{equation*}
		This finishes the proof that $\mathcal{R}$ is an $S^{G}$-linear splitting.
	\end{proof}

\section{The Reynolds operator for the classical actions} \label{sec:reynolds-classical}
	
	Fix an integer $t \ge 1$ and let $\GG(-)$ be one of $\GL_{t}(-)$, $\SL_{t}(-)$, $\OO_{t}(-)$, or $\Sp_{t}(-)$,
	where we assume that $t$ is even in the last case. 
	Define $C \coloneqq \GG(\mathbb{C}) \cap \UU_{t}(\mathbb{C})$. 
	The intersections in the respective cases are $\UU_{n}(\mathbb{C})$, $\SU_{n}(\mathbb{C})$, $\OO_{n}(\mathbb{R})$, and $\SpU_{n}(\mathbb{C})$.
	Let $k$ be an arbitrary field of characteristic zero. 

	\begin{thm}[The density theorem] \label{thm:density}
		With the above notation, we have:
		\begin{enumerate}[label=(\alph*)]
			\item $C$ is a Zariski-dense subgroup of $\GG(\mathbb{C})$; 
			and 
			\item $\GG(\mathbb{Q})$ is a Zariski-dense subgroup of $\GG(k)$.
		\end{enumerate}
	\end{thm}
	\begin{proof} 
		For (a), see the proof of~\Cite[Anhang II, Satz 4]{KraftGeometrische}. 
		We give a more elementary proof for $\GL$ and $\SL$ in \Cref{sec:proof-density}, see \Cref{prop:U-GL-dense,prop:SU-SL-dense}. 
		We also prove (b) in \Cref{sec:proof-density}, see \Cref{thm:G-Q-dense-in-G-k}.
	\end{proof}

	By $k[Z]$, we will mean one of $k[Y]$ or $k[X, Y]$. 
	In either case, we have a rational action of $\GG(k)$ on $k[Z]$, as described in \Cref{sec:classical-group-actions}. 
	Note that $C$ is a compact Lie group, and the action of $\GG(\mathbb{C})$ on $\mathbb{C}[Z]$ restricts to a smooth action of $C$.
	We have the following group extensions.

	\begin{equation*} 
		\begin{tikzcd}
		\GG(k) &                                                  & \GG(\mathbb{C}) &                                                          \\
		     & \GG(\mathbb{Q}) \arrow[lu, no head] \arrow[ru, no head] &               & C \arrow[lu, no head]
		\end{tikzcd}
	\end{equation*}

	We will first show how the computation of the Reynolds operator for $\GG(k)$ reduces to that for $C$. 
	The key point is that the action is rational, and each inclusion above is Zariski-dense by \Cref{thm:density}.
	This reduction is useful because $C$ is a compact Lie group; thus, we have its Reynolds operator by \Cref{thm:reynolds-for-lie-group}. 

	\begin{prop} \label{prop:same-invariants-upon-field-extension}
		Let $f_{1}, \ldots, f_{n} \in \mathbb{Q}[Z]^{\GG(\mathbb{Q})}$ be generating invariants, that is, we have
		$\mathbb{Q}[Z]^{\GG(\mathbb{Q})} = \mathbb{Q}[f_{1}, \ldots, f_{n}]$. 
		Then, the equality $k[Z]^{\GG(k)} = k[f_{1}, \ldots, f_{n}]$ holds. 
		In particular, we have the inclusion $\mathbb{Q}[Z]^{\GG(\mathbb{Q})} \subset k[Z]^{\GG(k)}$ as subsets of $k[Z]$.
	\end{prop}
	\begin{proof} 
		We first show that each $f_{i}$ is $\GG(k)$-invariant. 
		To this end, note that the equation
		\begin{equation*}
			\sigma(f_{i}) - f_{i} = 0
		\end{equation*}
		holds for each fixed $i$ and for all $\sigma \in \GG(\mathbb{Q})$. 
		Because the action is rational and $\GG(\mathbb{Q})$ is Zariski-dense in $\GG(k)$ by \Cref{thm:G-Q-dense-in-G-k}, the above equation must hold for all $\sigma \in \GG(k)$. 
		In other words, each $f_{i}$ is $\GG(k)$-invariant.

		We now prove the inclusion $k[Z]^{\GG(k)} \subset k[f_{1}, \ldots, f_{n}]$. 
		Let $B$ be a $\mathbb{Q}$-basis for $k$. 
		Given $h \in k[Z]^{\GG(k)}$, write 
		\begin{equation*} 
			h = \sum_{b \in B} b h_{b}
		\end{equation*}
		for $h_{b} \in \mathbb{Q}[Z]$. 
		If we apply $\sigma \in \GG(\mathbb{Q})$ to the above equation, we get
		\begin{equation*} 
			h = \sum_{b \in B} b \sigma(h_{b})
		\end{equation*}
		because $\sigma(h) = h$ and $\sigma(b) = b$ for all $b \in k$. 
		Comparing the two displayed equations above gives us that each $h_{b}$ is fixed by $\GG(\mathbb{Q})$ and thus
		$h_{b} \in \mathbb{Q}[f_{1}, \ldots, f_{n}]$ for all $b$. 
		In turn, $h \in k[f_{1}, \ldots, f_{n}]$, as desired.
	\end{proof}

	\begin{prop}
		Let $\mathcal{R}_{k} \colon k[Z] \onto k[Z]^{\GG(k)}$ denote the Reynolds operator over the field $k$. 
		The following diagram commutes
		\begin{equation*} 
			\begin{tikzcd}
			{k[Z]} \arrow[r, "\mathcal{R}_{k}", two heads]                                    & {k[Z]^{\GG(k)}}                                   \\
			{\mathbb{Q}[Z]} \arrow[r, "\mathcal{R}_{\mathbb{Q}}"', two heads] \arrow[u, hook] & {\mathbb{Q}[Z]^{\GG(\mathbb{Q})}}. \arrow[u, hook]
			\end{tikzcd}
		\end{equation*}
		In particular, if $\mu \in k[Z]$ is a monomial, then 
		\begin{equation} \label{eq:R-k-mu-R-C-mu}
			\mathcal{R}_{k}(\mu) = \mathcal{R}_{\mathbb{C}}(\mu).
		\end{equation}
	\end{prop}
	The above equation makes sense by interpreting $\mu$ as an element of $\mathbb{C}[Z]$. 

	\begin{proof} 
		In view of \Cref{prop:same-invariants-upon-field-extension}, 
		we may extend $\mathcal{R}_{\mathbb{Q}}$ $k$-linearly to obtain a retraction $\pi$ making the diagram
		\begin{equation*} 
			\begin{tikzcd}
			{k[Z]} \arrow[r, "\pi", two heads]                                    & {k[Z]^{\GG(k)}}                                   \\
			{\mathbb{Q}[Z]} \arrow[r, "\mathcal{R}_{\mathbb{Q}}"', two heads] \arrow[u, hook] & {\mathbb{Q}[Z]^{\GG(\mathbb{Q})}}. \arrow[u, hook]
			\end{tikzcd}
		\end{equation*}
		commute. 
		We need to show that $\pi = \mathcal{R}_{k}$. 
		By the uniqueness of the Reynolds operator, \Cref{thm:linearly-reductive-reynolds-unique-linear}, it suffices to show that $\pi$ is $\GG(k)$-equivariant. 
		Note that $\GG(k)$-equivariance can be checked on monomials, 
		where it is true again by the Zariski-density of $\GG(\mathbb{Q})$. 
		This proves that the diagram commutes.
		Now, if $\mu \in \mathbb{Q}[Y]$ is a monomial, then the diagram gives us $\mathcal{R}_{k}(\mu) = \mathcal{R}_{\mathbb{Q}}(\mu)$. 
		Because $k$ was arbitrary, we get~\Cref{eq:R-k-mu-R-C-mu}.
	\end{proof}

	The Zariski-density of $C$ in $\GG(\mathbb{C})$ similarly yields the following proposition.
	\begin{prop} \label{prop:invariants-and-operator-over-GC-and-intersection}
		The equality 
		$\mathbb{C}[Z]^{\GG(\mathbb{C})} 
		= 
		\mathbb{C}[Z]^{C}$ 
		holds, and 
		the splitting 
		$\mathcal{R} \colon \mathbb{C}[Z] \to \mathbb{C}[Z]^{C}$ 
		described in \Cref{thm:reynolds-for-lie-group} 
		is $\GG(\mathbb{C})$-equivariant. 
		In other words, $\mathcal{R}$ is the Reynolds operator for the $\GG(\mathbb{C})$-action, by \Cref{thm:linearly-reductive-reynolds-unique-linear}.
	\end{prop}

	\begin{rem}
		The above has now made the computation of $\mathcal{R}_{k}$ clear: because the Reynolds operator $\mathcal{R}_{k}$ is a $k$-linear map,
		it suffices to compute it on monomials; 
		and for monomials, $\mathcal{R}_{k}$ agrees with the Reynolds operator for the Lie group $C$ by~\Cref{eq:R-k-mu-R-C-mu} and \Cref{prop:invariants-and-operator-over-GC-and-intersection}. 
	\end{rem}

	In the following two subsections, we describe algorithms to implement this splitting on a computer algebra system. 

	\subsection{Computing the Reynolds operator for copies of the standard representation} \label{subsec:standard-computation}
	
	Continuing our notation from earlier, 
	let $\GG(k) \le \GL_{t}(k)$ be one of the classical groups, 
	and $C \coloneqq \GG(\mathbb{C}) \cap \UU_{t}(\mathbb{C})$ the corresponding compact Lie group. 
	For a positive integer $n$, the group $\GG(k)$ acts on $k[Y_{t \times n}]$ as described in~\ref{item:standard-action}.
	We describe the Reynolds operator for this action. 
	Consider the larger polynomial ring $k[Y][U_{t \times t}]$, and define the $k$-algebra map
	\begin{align*} 
		\phi \colon k[Y] &\to k[Y][U] \\
						  Y &\mapsto UY.
	\end{align*}

	For $f \in k[Y]$, write
	\begin{equation*} 
		\phi(f) = \sum_{I} \alpha_{I}(f) u^{I},
	\end{equation*}
	where $\alpha_{I}(f) \in k[Y]$;
	in the above, 
	the sum is over multi-indices $I \in \mathbb{N}^{t^{2}}$,
	and $u^{I}$ is the corresponding monomial.
	Each $u^{I}$ can be naturally interpreted as a 
	smooth function $C \to \mathbb{C}$ and
	the Reynolds operator is then given as
	\begin{equation} \label{eq:reynolds-standard-representation}
		\begin{aligned} 
			\mathcal{R} \colon k[Y] &\to k[Y]^{\GG(k)} \\
			f &\mapsto \sum_{I} \alpha_{I}(f) \int_{C} u^{I}.
		\end{aligned}
	\end{equation}

	\subsection{Computing the Reynolds operator for copies of the standard and the dual representations} \label{subsec:standard-dual-computation}

	We now consider the action of $\GG(k)$ on $k[X_{m \times t}, Y_{t \times n}]$ as described in~\ref{item:standard-dual-action}. 
	Note that while the action of $\GG(k)$ involves an inverse, 
	$C$ is a subgroup of the unitary group and thus, $U^{-1} = \overline{U}^{\tr}$ for $U \in C$. 
	We now consider the larger polynomial ring $k[X, Y][U_{t \times t}, \overline{U}_{t \times t}]$ 
	with $2t^{2}$ additional indeterminates; 
	explicitly, the new variables are the symbols 
	${\{u_{ij} : 1 \le i, j \le n\} \cup \{\overline{u}_{ij} : 1 \le i, j \le n\}}$. 
	Define the $k$-algebra map 
	\begin{align*} 
		\phi \colon k[X, Y] &\to k[X, Y][U, \overline{U}] \\
		X &\mapsto X \overline{U}^{\tr}, \\
		Y &\mapsto U Y.
	\end{align*}

	For $f \in k[X, Y]$, write
	\begin{equation*} 
		\phi(f) = \sum_{I, J} \alpha_{I, J}(f) u^{I} \overline{u}^{J}.
	\end{equation*}
	Each monomial $u^{I} \overline{u}^{J}$ can again be interpreted as a smooth function on $C$ and the Reynolds operator is given as
	\begin{equation} \label{eq:reynolds-standard-dual-representation}
		\begin{aligned} 
			\mathcal{R} \colon k[X, Y] &\to k[X, Y]^{\GG(k)} \\
			f &\mapsto \sum_{I, J} \alpha_{I, J}(f) \int_{C} u^{I} \overline{u}^{J}.
		\end{aligned}
	\end{equation}

	\subsection{Some remarks} \label{subsec:remarks}

	We stress that the only non-algebraic calculations above are 
	the integrals of monomial functions over $C$,
	where $C$ is one of $\UU_{t}(\mathbb{C})$, $\SU_{t}(\mathbb{C})$, $\OO_{t}(\mathbb{R})$, or $\SpU_{t}(\mathbb{C})$. 
	Note moreover that these are scalar functions. 
	While we discussed the theory of integration of vector-valued functions to prove the above, 
	one only needs to work with $\mathbb{C}$-valued functions in practice. 

	The integration of these monomial functions over $\UU_{t}(\mathbb{C})$, $\OO_{t}(\mathbb{R})$, and $\SpU_{t}(\mathbb{C})$ is of interest in various fields of mathematical physics, see the introduction of~\Cite{GorinLopez}. 
	Methods to compute these integrals are described in~\Cite{CollinsSniady, GorinLopez}. 
	In particular, the integration of arbitrary monomial functions over $\UU_{t}(\mathbb{C})$ has been implemented in the \texttt{Mathematica} package \texttt{IntU}~\Cite{PuchalaMiszczak}. 
	Using this package, we have implemented the splitting~\Cref{eq:reynolds-standard-dual-representation} for the action~\ref{item:standard-dual-action} of $\GL_{t}(\mathbb{C})$ in the computer algebra system \Sage~\Cite{sagemath}. 
	We have also implemented the splitting~\Cref{eq:reynolds-standard-representation} for the action~\ref{item:standard-action} of $\SL_{2}(\mathbb{C})$ using \Cref{thm:integrating-over-SU2}.

	For $\SL_{t}(k)$ and $\OO_{t}(k)$, the method described in \Cref{subsec:standard-dual-computation} for 
	the action~\ref{item:standard-dual-action} may be modified as follows.
	\begin{enumerate}[label=(\alph*)]
		\item (Special linear group) If $C = \SL_{t}(\mathbb{C}) \cap \UU_{t}(\mathbb{C})$, then the inverse of $U \in C$ is given by the adjugate $\adj(U)$. 
		Note that the entries of $\adj(U)$ are polynomials in the entries of $U$, so we may modify $\phi$ as
		\begin{align*} 
			\phi \colon k[X, Y] &\to k[X, Y][U] \\
			X &\mapsto X \adj(U), \\
			Y &\mapsto U Y.
		\end{align*}
		\item (Orthogonal group) If $C = \OO_{t}(\mathbb{C}) \cap \UU_{t}(\mathbb{C})$, then the inverse of $U \in C$ is just the transpose $U^{\tr}$, so we may modify $\phi$ as
		\begin{align*} 
			\phi \colon k[X, Y] &\to k[X, Y][U] \\
			X &\mapsto X U^{\tr}, \\
			Y &\mapsto U Y.
		\end{align*}
	\end{enumerate}

\section{Explicit formulae} \label{sec:explicit-formulae}
	In this section, we use the formulae of \Cref{sec:reynolds-classical} to compute the Reynolds operators for $\SL_{2}$ and $\GL_{t}$. 
	We give expressions for these in terms of the invariants described in \Cref{thm:classical-invariants}. 

	\subsection{The Reynolds operator for \texorpdfstring{$\SL_{2}$}{SL2}}

	We use formula~\Cref{eq:reynolds-standard-representation} to
	compute the Reynolds operator~$\mathcal{R}$ for the standard action~\ref{item:standard-action} of $\SL_{2}(k)$ on $k[Y_{2 \times N}]$; 
	the	relevant monomial integrals are determined in \Cref{thm:integrating-over-SU2} and we can thus compute $\mathcal{R}$ on any element of $k[Y]$. 
	We begin the section by recording the value of $\mathcal{R}$ on various families of monomials, postponing the proofs until the end of the section. 
	By \Cref{thm:classical-invariants}, we know that $k[Y]^{\SL_{2}(k)}$ is generated by the size $2$ minors of $Y$. 
	For ease of notation, we write
	\begin{equation*} 
		Y = 
		\begin{bmatrix}
			a_{1} & a_{2} & \cdots & a_{N} \\
			b_{1} & b_{2} & \cdots & b_{N} \\
		\end{bmatrix}
		,
		\qquad
		\{\Delta\} \coloneqq \{\text{size $2$ minors of $Y$}\},
		\qquad
		\text{and}
		\qquad
		\Delta_{i, j} \coloneqq a_{i} b_{j} - a_{j} b_{i}.
	\end{equation*}

	The next theorem describes the Reynolds operator on $k[Y_{2 \times 2}]$.

	\begin{thm} \label{thm:reynolds-operator-SL-2-by-2}
		Let $\mathcal{R} \colon k[Y_{2 \times 2}] \to k[\{\Delta\}]$ 
		be the Reynolds operator and $\mu \in k[Y_{2 \times 2}]$ a monomial.
		\begin{enumerate}[leftmargin=*, label=(\alph*)]
			\item If $\mu$ is of the form $(a_{1} b_{2})^{n} (a_{2} b_{1})^{m}$ for some nonnegative integers $n$ and $m$, then
			\begin{equation} \label{eq:R-SL-2-2}
				\mathcal{R}(\mu) 
				= \mathcal{R}\left((a_{1} b_{2})^{n} (a_{2} b_{1})^{m}\right) 
				=	\frac{n! m!}{(n + m + 1)!} \Delta_{1,2}^{n} \Delta_{2,1}^{m};
			\end{equation}
			in particular, for $n \ge 0$, we have
			\begin{equation} \label{eq:R-SL-2-1}
				\mathcal{R}\left((a_{1} b_{2})^{n}\right) 
				= \frac{1}{n + 1}\Delta_{1,2}^{n}.
			\end{equation}
			\item If $\mu$ is not of the above form, then
			\begin{equation*} 
				\mathcal{R}(\mu) = 0.
			\end{equation*}
		\end{enumerate}
	\end{thm}

	We give $k[Y_{2 \times N}]$ a multi-grading by defining $\deg(a_{i}) = (1, 0)$ and $\deg(b_{i}) = (0, 1)$ for all $1 \le i \le N$.

	\begin{thm} \label{thm:row-unbalanced-in-kernel}
		Let $\mu \in k[Y]$ be a monomial such that $\deg(\mu) = (m, n)$ with $m \neq n$. 
		Then, $\mathcal{R}(\mu) = 0$.
	\end{thm}

	Computations suggest that~\Cref{eq:R-SL-2-2} generalises as follows.

	\begin{conj} \label{conj:2x3-formula}
		For all nonnegative integers $i$, $j$, $k$, we have
		\begin{equation*} 
			\mathcal{R}\left(
			(a_{1} b_{2})^{i}
			(a_{1} b_{3})^{j}
			(a_{2} b_{3})^{k}
			\right)
			=
			\frac{(i + j)! (k + j)!}{(i + j + k + 1)! j!}
			\Delta_{1, 2}^{i} \Delta_{1, 3}^{j} \Delta_{2, 3}^{k}.
		\end{equation*}
	\end{conj}

	\begin{conj} \label{conj:odd-powers-in-kernel}
		For all nonnegative integers $n$, we have
		\begin{equation*} 
			\mathcal{R}\left((a_{1} a_{2} a_{3} b_{1} b_{2} b_{3})^{2n + 1}\right) = 0.
		\end{equation*}
	\end{conj}

	\begin{thm} \label{thm:integrating-over-SU2}
		For all nonnegative integers $a$, $b$, $c$, $d$, we have
		\begin{equation*} 
			\int_{\SU_{2}(\mathbb{C})} u_{11}^{a} u_{12}^{b} u_{21}^{c} u_{22}^{d} = 
			\begin{cases}
				(-1)^{b} \dfrac{a! b!}{(a + b + 1)!} & \text{if $a = d$ and $b = c$}, \\[3pt]
				0 & \text{else}.
			\end{cases}
		\end{equation*}
	\end{thm}
	\begin{proof}
		See \Cref{thm:integrating-over-SU2-appendix}.
	\end{proof}

	We say that a monomial in $k[Y]$ is \deff{balanced} if it is a product of monomials of the form $a_{i} b_{j}$ with $i \neq j$, and \deff{unbalanced} otherwise. 
	The following are straightforward observations.
	\begin{enumerate}[label=(\alph*)]
		\item The algebra of minors $k[\{\Delta\}]$ sits inside the $k$-subalgebra generated by the balanced monomials.
		\item If $\mu \in k[Y]$ is a balanced monomial, then $\deg(\mu) = (d, d)$ for some $d \ge 0$.
	\end{enumerate}
	Note however that $\deg(a_{1} b_{1}) = (1, 1)$, yet $a_{1} b_{1}$ is unbalanced. 

	\begin{rem}
		Assuming \Cref{conj:2x3-formula}, the $k[\{\Delta\}]$-linearity of $\mathcal{R}$
		would then determine the value of $\mathcal{R}$ on any balanced monomial in $k[Y_{2 \times 3}]$.
		For example, one may verify \Cref{conj:2x3-formula} in the two cases needed for the following computation and obtain
		\begin{equation*} 
			\mathcal{R}\left((a_{1} b_{2})(a_{2} b_{3})(a_{3} b_{1})\right) 
			= 
			\mathcal{R}\left((a_{1} b_{2})(a_{2} b_{3})(a_{1} b_{3} - \Delta_{1, 3})\right) = \left(\frac{1}{6} - \frac{1}{6}\right) \Delta_{1, 2} \Delta_{1, 3} \Delta_{2, 3} = 0.
		\end{equation*}
		The above gives us the case $n = 0$ in \Cref{conj:odd-powers-in-kernel}. 
		In particular, it shows that a balanced monomial may be in the kernel of $\mathcal{R}$; 
		\Cref{thm:reynolds-operator-SL-2-by-2} tells us that this does not happen for $k[Y_{2 \times 2}]$, 
		where the monomials in the kernel are precisely the unbalanced ones. 
	\end{rem}

	\begin{rem}
		It is not true that the image of a monomial is again a monomial in the $\Delta_{i, j}$. One checks that
		\begin{equation*} 
			\mathcal{R}(a_{1} b_{2} a_{3} b_{4}) 
			= 
			\frac{1}{3} \Delta_{1, 2} \Delta_{3, 4} - \frac{1}{6} \Delta_{1, 3} \Delta_{2, 4}.
		\end{equation*}
		The expression on the right is not divisible by any $\Delta_{i, j}$ and thus cannot be expressed as a monomial in the $\{\Delta\}$.
	\end{rem}

	\begin{proof}[Proof of \Cref{thm:reynolds-operator-SL-2-by-2} (a)] 
		The map $\phi$ from \Cref{subsec:standard-computation} is given by
		\begin{equation*} 
			\begin{bmatrix}
			a_{1} & \cdots & a_{N} \\
			b_{1} & \cdots & b_{N} \\
		\end{bmatrix}
		\mapsto
		\begin{bmatrix}
			u_{11} & u_{12} \\
			u_{21} & u_{22}
		\end{bmatrix}
		\begin{bmatrix}
			a_{1} & \cdots & a_{N} \\
			b_{1} & \cdots & b_{N} \\
		\end{bmatrix}.
		\end{equation*}
		Thus, 
		\begin{align*} 
			\phi(a_{1}) &= a_{1} u_{11} + b_{1} u_{12}, \ \text{and} \\
			\phi(b_{2}) &= a_{2} u_{21} + b_{2} u_{22}.
		\end{align*}
		Because $\phi$ is a ring homomorphism, we have
		\begin{align*} 
			\phi\left((a_{1}b_{2})^{n}\right) 
			&= \sum_{i + j + k + \ell = n} 
			\binom{n}{i, j, k, \ell}
			\left(a_{1} a_{2} u_{11} u_{21}\right)^{i} 
			\left(a_{1} b_{2} u_{11} u_{22}\right)^{j}
			\left(a_{2} b_{1} u_{12} u_{21}\right)^{k}
			\left(b_{1} b_{2} u_{12} u_{22}\right)^{\ell}.
		\end{align*}

		\Cref{thm:integrating-over-SU2} tells us that if we integrate the above over $\SU_{2}(\mathbb{C})$, 
		the only terms that remain are those with $i = \ell$. 
		Integrating those terms, we get

		\begin{align*} 
			\mathcal{R}\left((a_{1}b_{2})^{n}\right) 
			&=
			\sum_{2i + j + k = n} 
			\binom{n}{i, j, k, i}
			(a_{1} b_{2})^{i + j} (a_{2} b_{1})^{i + k}
			(-1)^{i + k} \frac{(i + j)! (i + k)!}{(n + 1)!} \\
			&= \frac{1}{n + 1}(a_{1} b_{2} - a_{2} b_{1})^{n} 
			= \frac{\Delta_{1, 2}^{n}}{n + 1},
		\end{align*}

		where the penultimate equality uses \Cref{identity:x-plus-y-multinomial}, proving~\Cref{eq:R-SL-2-1}. 
		For~\Cref{eq:R-SL-2-2}, note that $a_{2} b_{1} = a_{1} b_{2} + \Delta_{2, 1}$. 
		Because $\mathcal{R}$ is $k[\{\Delta\}]$-linear and $\Delta_{1, 2} = -\Delta_{2, 1}$, we get
		\begin{align*} 
			\mathcal{R}((a_{1} b_{2})^{n} (a_{2} b_{1})^{m}) 
			&= \mathcal{R}\left(
			(a_{1} b_{2})^{n} (a_{1} b_{2} + \Delta_{2, 1})^{m}
			\right) \\
			&= \sum_{k = 0}^{m} \binom{m}{k} \Delta_{2, 1}^{m - k} 
			\mathcal{R}\left(
			(a_{1} b_{2})^{n + k}
			\right) \\
			&= \sum_{k = 0}^{m} \binom{m}{k} \Delta_{2, 1}^{m - k} \cdot \frac{\Delta_{1, 2}^{n + k}}{n + k + 1} \\
			&= \Delta_{1, 2}^{n} \Delta_{2, 1}^{m} \sum_{k = 0}^{m} \binom{m}{k} \frac{(-1)^{k}}{n + k + 1}.
		\end{align*}
		\Cref{identity:2} finishes the proof.
	\end{proof}

	\begin{proof}[Proof of \Cref{thm:row-unbalanced-in-kernel}]
		Consider the element 
		$\sigma = \left(
		\begin{smallmatrix}
		2 & 0 \\ 
		0 & 2^{-1} \\
		\end{smallmatrix}
		\right) \in \SL_{2}(k)$. 
		We have $\sigma(\mu) = 2^{m - n} \mu$ and thus, 
		the $\SL_{2}(k)$-equivariance of $\mathcal{R}$ implies that $\mathcal{R}(\mu) = 2^{m - n} \mathcal{R}(\mu)$.
		Because $m \neq n$, we get $\mathcal{R}(\mu) = 0$.
	\end{proof}

	\begin{proof}[Proof of \Cref{thm:reynolds-operator-SL-2-by-2} (b)]  
		We first prove the statement when $\mu$ is of the form $(a_{1} b_{1})^{m} (a_{1} b_{2})^{n}$ for some $m > 0$ and $n \ge 0$. 
		We have
		\begin{align*} 
			\phi(\mu) &= (a_{1}^{2} u^{\ast} + a_{1} b_{1} u^{\ast} + a_{1} b_{1} u^{\ast} + b_{1}^{2} u^{\ast})^{m} \cdot (a_{1} a_{2} u^{\ast} + a_{1} b_{2} u^{\ast} + a_{2} b_{1} u^{\ast} + b_{1} b_{2} u^{\ast})^{n},
		\end{align*}
		where each $u^{\ast}$ denotes some monomial in the $u_{ij}$. 
		Because $m > 0$, when we expand the above, each monomial that appears will be unbalanced in the sense that we may write
		\begin{equation*} 
			\phi(\mu) = \sum_{I} \alpha_{I} \mu_{I} u_{I},
		\end{equation*}
		where $\alpha_{I} \in k$, $\mu_{I} \in k[Y]$ is an unbalanced monomial, and $u_{I} \in k[U]$ is a monomial. 
		Integrating the above yields
		\begin{equation*} 
			\mathcal{R}(\mu) = \sum_{I} \left(\alpha_{I} \textstyle \int\! u_{I}\right) \mu_{I}. 
		\end{equation*}
		Now, note that $\mathcal{R}(\mu) \in k[\{\Delta\}] \subset k[\text{balanced monomials}]$, whereas each $\mu_{I}$ above is unbalanced. 
		Thus, the terms above must cancel out to give us $\mathcal{R}(\mu) = 0$.

		The $k[\{\Delta\}]$-linearity of $\mathcal{R}$ then implies the statement for $\mu$ of the form $(a_{1} b_{1})^{m} \nu$ with $m > 0$ and $\nu \in k[a_{1} b_{2}, a_{2} b_{1}]$. 
		By symmetry, the statement also holds for $\mu$ of the form $(a_{2} b_{2})^{m} \nu$. 
		\Cref{thm:row-unbalanced-in-kernel} takes care of unbalanced monomials not of the above form.
	\end{proof}

	\subsection{The Reynolds operator for \texorpdfstring{$\GL_{t}$}{GLt}}

	Let $t$, $n$, $m$ be positive integers,
	and $\mathcal{R} \colon k[X_{m \times t}, Y_{t \times n}] \to k[X, Y]^{\GL_{t}(k)}$ the Reynolds operator for the action~\ref{item:standard-dual-action}. 
	By \Cref{thm:classical-invariants}, we know the image of $\mathcal{R}$ to lie in $k[XY]$, the subalgebra of $k[X, Y]$ generated by the entries of $XY$. 
	Experimenting with the package \texttt{IntU}~\Cite{PuchalaMiszczak} suggests a formula similar to~\Cref{eq:R-SL-2-1}. 

	\begin{conj}
		For $t = 2$ and $n \ge 0$, we have
		\begin{equation*} 
			\mathcal{R}\left((x_{11} y_{11})^{n}\right) 
			= 
			\frac{1}{n + 1}(x_{11} y_{11} + x_{12} y_{21})^{n}
			=
			\frac{1}{n + 1}\left([XY]_{1, 1}\right)^{n}.
		\end{equation*}

		More generally, for $t \ge 1$ and $n \ge 0$, we have

		\begin{equation*} 
			\mathcal{R}\left((x_{11} y_{11})^{n}\right) 
			= 
			\binom{n + t - 1}{t - 1}^{-1}\left([XY]_{1, 1}\right)^{n}.
		\end{equation*}
	\end{conj}

\section{Comparison with positive characteristic} \label{sec:positive-characteristic}
	
	The classical groups $\GL$, $\SL$, $\OO$, $\Sp$ are typically 
	\emph{not} linearly reductive in positive characteristic. 
	Thus, there is no guarantee of the existence of splittings that are linear over the fixed subring. 
	In fact, the following theorem tells us that this is essentially never the case.

	\begin{thm}[{\Cite[Theorem 1.1]{HochsterJeffriesPandeySingh}}] \label{thm:HJPS}
		Let $k$ be a field of characteristic $p > 0$. 
		Fix positive integers $m$, $n$, and $t$, and let $R \subset S$ denote one of the following inclusions:
		\begin{enumerate}[label=(\alph*)]
			\item $k[XY] \subset k[X_{m \times t}, Y_{t \times n}]$;
			\item $k[\{\Delta\}] \subset k[Y_{t \times n}]$ with $t \le n$, where $\{\Delta\}$ is the set of size $t$ minors of $Y$;
			\item $k[Y^{\tr} Y] \subset k[Y_{t \times n}]$;
			\item $k[Y^{\tr} \Omega Y] \subset k[Y_{2t \times n}]$.
		\end{enumerate}

		Then the inclusion $R \subset S$ splits $R$-linearly if and only if, in the respective cases,
		\begin{enumerate}[label=(\alph*)]
			\item $t = 1$ or $\min\{m, n\} \le t$;
			\item $t = 1$ or $t = n$;
			\item $t = 1$; $t = 2$ and $p$ is odd; $p = 2$ and $n \le (t + 1)/2$; or $p$ is odd and $n \le (t + 2)/2$;
			\item $n \le t + 1$.
		\end{enumerate}
	\end{thm}

	\begin{rem}
		The above theorem does not reference any group (action). 
		However, compare with \Cref{thm:classical-invariants} to see the connection for infinite fields of positive characteristic. 
	\end{rem}

	\begin{rem} \label{rem:primes-in-denominators}
		We describe a curious implication of \Cref{thm:HJPS}. We revisit formula~\Cref{eq:R-SL-2-1}: 
		\begin{equation*} 
			\mathcal{R}\left((a_{1} b_{2})^{n}\right) = \frac{1}{n + 1}\Delta_{1,2}^{n}.
		\end{equation*}
		Note the denominator `$n + 1$'. 
		This means that each prime number shows up as a factor of the denominator for some monomial.
		Said differently, $\mathcal{R}$ does not restrict to a map $\mathbb{Z}_{(p)}[Y] \to \mathbb{Z}_{(p)}[\{\Delta\}]$ for any prime $p > 0$,
		where $\mathbb{Z}_{(p)}$ is the subring of $\mathbb{Q}$ defined as
		\begin{equation*} 
			\mathbb{Z}_{(p)} \coloneqq \left\{\frac{a}{b} \in \mathbb{Q} : a, b \in \mathbb{Z} \ \text{with} \ p \nmid b\right\}.
		\end{equation*}
		\Cref{thm:HJPS} tells us that this must essentially always happen for any of the Reynolds operators described in the paper. 
		More generally, the above must happen for essentially any splitting that is linear over the subring. 

		Indeed, pick a situation in \Cref{thm:HJPS} where the inclusion does not split in positive characteristic. 
		For example, $\mathbb{F}_{p}[\{\Delta\}] \into \mathbb{F}_{p}[Y_{t \times n}]$ with $1 < t < n$. 
		As discussed earlier, the inclusion $\mathbb{Q}[\{\Delta\}] \into \mathbb{Q}[Y_{t \times n}]$ \emph{does} split. 
		Moreover, if we are only interested in splittings that are linear over the subring, 
		then there are typically more than one. 
		Let $\pi \colon \mathbb{Q}[Y_{t \times n}] \to \mathbb{Q}[\{\Delta\}]$ be any such $\mathbb{Q}[\{\Delta\}]$-linear splitting.
		The following must hold: given any prime $p > 0$, 
		there exists some monomial $\mu = \mu(p) \in \mathbb{Q}[Y]$ such that 
		when we express $\pi(\mu)$ as a polynomial in the $\{\Delta\}$ with rational coefficients, 
		then one of the coefficients has denominator divisible by $p$. 
		Indeed, if this were not the case for some prime $p$, 
		then $\pi$ would restrict to a splitting $\mathbb{Z}_{(p)}[Y] \to \mathbb{Z}_{(p)}[\{\Delta\}]$,
		and we could go mod $p$ to obtain an $\mathbb{F}_{p}[\{\Delta\}]$-linear splitting, 
		contradicting \Cref{thm:HJPS}.
	\end{rem}

\appendix

\section{Proof of the density theorem} \label{sec:proof-density}

	\begin{defn}
		For $X$ a topological space and $Y$ a subspace of $X$, 
		a \deff{retraction} of $X$ onto $Y$ is a continuous function $r \colon X \to Y$ 
		satisfying $r(y) = y$ for all $y \in Y \subset X$.
	\end{defn}

	\begin{lem} \label{lem:vanishing-on-product-infinite-sets}
		Let $k$ be a field, 
		and $S \subset k$ be an infinite subset. 
		If $f \in k[x_{1}, \ldots, x_{n}]$ is a polynomial vanishing on
		the product $S^{n} \subset \mathbb{A}_{k}^{n}$, 
		then $f$ is the zero polynomial. 
		Equivalently, $S^{n}$ is Zariski-dense in $\mathbb{A}_{k}^{n}$.
	\end{lem}
	\begin{proof}
		We prove the statement by induction on $n$. 
		It is clear for $n = 1$. 
		Assume $n > 1$ and suppose $f$ is nonzero. 
		Write
		$f = f_{0} + f_{1} x_{n} + \cdots + f_{d} x_{n}^{d}$ 
		with $d \ge 0$, 
		$f_{d} \neq 0$,
		and $f_{i} \in k[x_{1}, \ldots, x_{n - 1}]$. 
		By induction, there exists 
		$\mathbf{s} = (s_{1}, \ldots, s_{n - 1}) \in S^{n - 1}$ with 
		$f_{d}(\mathbf{s}) \neq 0$. 
		Then, $f(\mathbf{s}, x_{n})$ is a nonzero polynomial in one variable, and this finishes the proof. 
	\end{proof}

	\begin{lem} \label{lem:dense-intersection-dense}
		Let $X$ be a topological space, 
		$Z \subset X$ a dense subspace, 
		and $Y \subset X$ a subspace such that there exists a 
		retraction $r \colon X \onto Y$ with $r(Z) \subset Z$. 
		Then, $Z \cap Y$ is dense in $Y$.
	\end{lem}
	\begin{proof} 
		Let $y \in Y$ be arbitrary. 
		As $Z$ is dense in $X$, 
		there exists 
		a net $\langle z_{\lambda} \rangle_{\lambda \in \Lambda}$ in $Z$
		with $z_{\lambda} \to y$. 
		In turn, $\langle r(z_{\lambda}) \rangle_{\lambda}$ is a net in
		$Z \cap Y$ converging to $y$.
	\end{proof}

	For the next few proofs, we define the function
	\begin{equation} \label{eq:retraction-GL-SL}
		\begin{aligned} 
			r \colon \GL_{n}(k) & \to \SL_{n}(k) \\
			U = 
			\begin{bmatrix}
				u_{11} & u_{12} & \cdots & u_{1n} \\
				u_{21} & u_{22} & \cdots & u_{2n} \\
				\vdots & \vdots & \ddots & \vdots \\
				u_{n1} & u_{n2} & \cdots & u_{nn}
			\end{bmatrix}
			&\mapsto 
			\begin{bmatrix}
				\frac{u_{11}}{\det U} & \frac{u_{12}}{\det U} & \cdots & \frac{u_{1n}}{\det U} \\
				u_{21} & u_{22} & \cdots & u_{2n} \\
				\vdots & \vdots & \ddots & \vdots \\
				u_{n1} & u_{n2} & \cdots & u_{nn}
			\end{bmatrix}.
		\end{aligned}
	\end{equation}
	That is, $r$ scales the first row of $U$ by $\frac{1}{\det U}$.

	\begin{lem} 
		For any field $k$, 
		the function $r$ defined by~\Cref{eq:retraction-GL-SL} 
		is a retraction of 
		$\GL_{n}(k)$ onto $\SL_{n}(k)$. 
	\end{lem}
	\begin{proof} 
		The multilinearity of $\det$ implies 
		that $\det(r(U)) = 1$ for all $U \in \GL_{n}(\mathbb{C})$, 
		that is, 
		$r$ indeed takes values in $\SL_{n}(k)$. 
		The function $r$ is continuous in the Zariski topology 
		because it is given by rational functions.
		It is clear that the restriction of $r$ to $\SL_{n}(k)$ is the identity.
	\end{proof}

	\begin{prop} \label{prop:U-GL-dense}
		For all $n \ge 1$, the subgroup 
		$\UU_{n}(\mathbb{C})$ is Zariski-dense in $\GL_{n}(\mathbb{C})$.
	\end{prop}
	\begin{proof} 
		Let $C$ be the Zariski closure of $\UU_{n}(\mathbb{C})$ in $\GL_{n}(\mathbb{C})$. 
		Write $C = V(\mathfrak{a}) \cap \GL_{n}(\mathbb{C})$ for $\mathfrak{a}$ an ideal. 
		Let $f \in \mathfrak{a}$. 
		Note that if $z_{1}, \ldots, z_{n}$ are elements of the unit circle $\mathbb{S}^{1}$, 
		then $\diag(z_{1}, \ldots, z_{n})$ is an element of $\UU_{n}$. 
		Thus, $f$ vanishes on all diagonal matrices with entries coming from $\mathbb{S}^{1}$. 
		By \Cref{lem:vanishing-on-product-infinite-sets}, we see that $f$ must vanish on \emph{all} diagonal matrices. 
		Thus, $C$ contains all invertible diagonal matrices.

		Because $\GL_{n}(\mathbb{C})$ is a topological group in the Zariski topology, 
		and $\UU_{n}(\mathbb{C})$ is a subgroup, 
		it follows that $C$ is a subgroup. 
		As every invertible matrix can be decomposed as $U D V$ with 
		$U, V \in \UU_{n}(\mathbb{C})$ and $D$ invertible diagonal, 
		we are done.
	\end{proof}

	\begin{prop} \label{prop:SU-SL-dense}
		For all $n \ge 1$, the subgroup 
		$\SU_{n}(\mathbb{C})$ is Zariski-dense in $\SL_{n}(\mathbb{C})$.
	\end{prop}
	\begin{proof} 
		We use \Cref{lem:dense-intersection-dense} with 
		$X = \GL_{n}(\mathbb{C})$, 
		$Z = \UU_{n}(\mathbb{C})$, 
		$Y = \SL_{n}(\mathbb{C})$, 
		and $r$ given by~\Cref{eq:retraction-GL-SL}.
		The density of $Z$ then follows from \Cref{prop:U-GL-dense}. 
		All that is left to be shown is that 
		$r(\UU_{n}(\mathbb{C})) \subset \UU_{n}(\mathbb{C})$. 
		To this end,
		note that a matrix is unitary 
		if and only if 
		its rows form an orthonormal basis. 
		If $U \in \UU_{n}(\mathbb{C})$, 
		then $\det(U) \in \mathbb{S}^{1}$
		and thus, the rows of $r(U)$ continue to be orthonormal.
	\end{proof}	

	\begin{thm} \label{thm:G-Q-dense-in-G-k}
		Let $k$ be a field of characteristic zero. 
		For each of the following inclusions, the subgroup is Zariski-dense in the larger group.
		\begin{enumerate}[label=(\alph*)]
			\item $\GL_{n}(\mathbb{Q}) \subset \GL_{n}(k)$,
			\item $\SL_{n}(\mathbb{Q}) \subset \SL_{n}(k)$,
			\item $\OO_{n}(\mathbb{Q}) \subset \OO_{n}(k)$, and
			\item $\Sp_{2n}(\mathbb{Q}) \subset \Sp_{2n}(k)$.
		\end{enumerate}
	\end{thm}
	\begin{proof}
		General linear group: 
		By \Cref{lem:vanishing-on-product-infinite-sets}, 
		the subspace $\mathbb{Q}^{n^{2}}$ is dense in $\mathbb{A}_{k}^{n^{2}}$. 
		Intersecting with the open set $\GL_{n}(k)$ gives us (a).

		Special linear group: (b) then follows by use of \Cref{lem:dense-intersection-dense} 
		with $X = \GL_{n}(k)$, 
		$Y = \SL_{n}(k)$,
		$Z = \GL_{n}(\mathbb{Q})$, and 
		$r$ given by~\Cref{eq:retraction-GL-SL}.

		Orthogonal group: We note that the orthogonal group $\OO_{n}(k)$ is generated by the set of reflections
		\begin{equation*} 
			R(k) \coloneqq \left\{I - \frac{2 u u^{\tr}}{u^{\tr} u} : \text{$u \in k^{n}$ with $u^{\tr} u \neq 0$}\right\},
		\end{equation*}
		in fact the Cartan--Dieudonn\'e theorem states that every orthogonal matrix is a product of at most $n$ such reflections, see~\Cite{DieudonneGroupesClassiquesOG, ScherkOrthogonal}. 
		Because the closure of $\OO_{n}(\mathbb{Q})$ must be a subgroup of $\OO_{n}(k)$, it suffices to show that $R(\mathbb{Q})$ is dense in $R(k)$.
		To this end, note that $I(k) \coloneqq \{u \in k^{n} \colon u^{\tr} u \neq 0\}$ is an open subset of $\mathbb{A}_{k}^{n}$ and thus intersecting with the dense set $\mathbb{Q}^{n}$, 
		we get that $I(\mathbb{Q})$ is dense in $I(k)$. 
		Now, $R(k)$ is the image of $I(k)$ under the continuous map $u \mapsto I - \frac{2 u u^{\tr}}{u^{\tr} u}$ and hence $R(\mathbb{Q})$ is dense in $R(k)$.

		Symplectic group: (d) follows similarly by using the fact that the symplectic group $\Sp_{2n}(k)$ is generated by 
		\begin{equation*} 
			\begin{bmatrix}
				A & O \\
				O & (A^{\tr})^{-1} \\
			\end{bmatrix}, 
			\qquad
			\begin{bmatrix}
				I & B \\
				O & I \\
			\end{bmatrix},
			\quad
			\text{and}
			\quad
			\begin{bmatrix}
				O & I \\
				-I & O \\
			\end{bmatrix},
		\end{equation*}
		where $A$ varies over $\SL_{n}(k)$, 
		and $B$ over all symmetric $n \times n$ matrices. 
		This description is originally due to \Citeauthor{DieudonneGenerators}~\Cite{DieudonneGenerators} and can also be found in~\Cite[\S2.2]{OMearaSymplectic}.
	\end{proof}

\section{Multinomial coefficient and integration identities} \label{sec:identities}

	\begin{identity} \label{identity:beta-integral}
		For integers $a, b \ge 0$, we have
		\begin{equation*}
			\int_{0}^{1} t^{a} (1 - t)^{b} \,{\mathrm{d}}t = \frac{a! b!}{(a + b + 1)!}.
		\end{equation*}
	\end{identity}
	\begin{proof}
		The formula is readily verified if $b = 0$. 
		For $a \ge 0$ and $b > 0$, integration by parts yields
		\begin{equation*} 
			\int_{0}^{1} t^{a} (1 - t)^{b} \,{\mathrm{d}}t = \frac{b}{a + 1} \int_{0}^{1} t^{a + 1} (1 - t)^{b - 1} \,{\mathrm{d}}t.
		\end{equation*}
		Repeated application of the above gives the desired formula.
	\end{proof}

	\begin{identity} \label{identity:x-plus-y-multinomial}
		Let $n \ge 0$ be an integer. One has the identity
		\begin{equation*} 
			\frac{(x + y)^{n}}{n + 1} 
			= 
			\sum_{2i + j + k = n} 
			\binom{n}{i, i, j, k} 
			\frac{(i + j)! (i + k)!}{(n + 1)!} 
			x^{i + j} y^{i + k},
		\end{equation*}
		where, explicitly, the sum is taken over all triples $(i, j, k) \in \mathbb{N}^{3}$ satisfying $2i + j + k = n$.
	\end{identity}
	\begin{proof} 
		Note that
		\begin{equation*} 
			\binom{n}{i, i, j, k} \frac{(i + j)! (i + k)!}{(n + 1)!} = \frac{n!}{i! i! j! k!} \frac{(i + j)! (i + k)!}{(n + 1)!} = \frac{1}{n + 1} \binom{i + j}{i} \binom{i + k}{k}.
		\end{equation*}
		Thus, the identity of interest is equivalent to
		\begin{equation*} 
			(x + y)^{n} = 
			\sum_{2i + j + k = n} 
			\binom{i + j}{i} \binom{i + k}{k}
			x^{i + j} y^{i + k}.
		\end{equation*}
		Because both sides of the equation are homogeneous of degree $n$, 
		it suffices to verify that
		\begin{align} \label{eq:001} 
			(x + 1)^{n} = 
			\sum_{2i + j + k = n} 
			\binom{i + j}{i} \binom{i + k}{k}
			x^{i + j}.
			\tag{$\star$}
		\end{align}

		To prove the above identity, we need to show that the coefficient of $x^{a}$ is the same on both sides for each
		$0 \le a \le n$.
		The coefficient of $x^{a}$ on the right-hand-side of~\Cref{eq:001} is given by
		\begin{align*} 
			\sum_{\substack{2i + j + k = n \\ i + j = a}}
			\binom{i + j}{i} \binom{i + k}{k} 
			= \sum_{i + k = n - a} \binom{a}{i} \binom{i + k}{i}
			= \sum_{i} \binom{a}{i} \binom{n - a}{i}.
		\end{align*}

		Thus, it suffices to prove that
		\begin{align} \label{eq:002}
			\binom{n}{a} = \sum_{i} \binom{a}{i} \binom{n - a}{i}.
			\tag{$\dagger$}
		\end{align}

		To this end, note that
		\begin{equation*} 
			(1 + X)^{a} (1 + Y)^{n - a} = \sum_{i, j} \binom{a}{i} \binom{n - a}{j} X^{i} Y^{j}.
		\end{equation*}

		Substituting $Y = 1/X$ gives
		\begin{equation*} 
			(1 + X)^{a} \left(1 + \frac{1}{X}\right)^{n - a} 
			= 
			\sum_{i, j} \binom{a}{i} \binom{n - a}{j} X^{i - j}.
		\end{equation*}

		Thus,
		\begin{equation*} 
			\frac{1}{X^{n - a}} (1 + X)^{n}
			= 
			\sum_{i, j} \binom{a}{i} \binom{n - a}{j} X^{i - j}.
		\end{equation*}

		Comparing the coefficient of $X^{0}$ on both sides gives us~\Cref{eq:002}.
	\end{proof}

	\begin{identity} \label{identity:2}
		For integers $m, n \ge 0$, one has the identity
		\begin{equation*} 
			\sum_{k = 0}^{n} \binom{n}{k} \frac{(-1)^{k}}{m + k + 1} = \frac{m! n!}{(m + n + 1)!}.
		\end{equation*}
	\end{identity}
	\begin{proof} 
		We note
		\begin{align*} 
			\sum_{k = 0}^{n} \binom{n}{k} \frac{(-1)^{k}}{m + k + 1} &= \sum_{k = 0}^{n} \binom{n}{k} (-1)^{k} \int_{0}^{1} t^{m + k} \,{\mathrm{d}}t \\
			&= \int_{0}^{1} t^{m} \cdot \sum_{k = 0}^{n} \binom{n}{k} (-t)^{k} \,{\mathrm{d}}t \\
			&= \int_{0}^{1} t^{m} (1 - t)^{n} \,{\mathrm{d}}t \\
			&= \frac{m! n!}{(m + n + 1)!},
		\end{align*}
		where the last step uses \Cref{identity:beta-integral}.
	\end{proof}

	\begin{identity} \label{identity:integrate-cos-sin}
		For integers $a, b \ge 0$, we have
		\begin{equation*} 
			\int_{0}^{\pi/2} 
			\cos^{2a}(\theta) \sin^{2b}(\theta) \sin(2\theta)
			\,{\mathrm{d}}\theta 
			= \frac{a!b!}{(a + b + 1)!}.
		\end{equation*}
	\end{identity}
	\begin{proof} 
		The integrand can be rewritten as
		\begin{align*}
			\cos^{2a}(\theta) \sin^{2b}(\theta) \sin(2\theta) 
			&= 2 \cos^{2a + 1}(\theta) \sin^{2b + 1}(\theta) \\
			&= 2 (\cos^{2}(\theta))^{a} (\sin(\theta))^{2b + 1} \cos(\theta) \\
			&= 2 (1 - \sin^{2}(\theta))^{a} (\sin(\theta))^{2b + 1} \cos(\theta).
		\end{align*}

		The substitution $u = \sin(\theta)$ gives us
		\begin{align*} 
			\int_{0}^{\pi/2} 
			\cos^{2a}(\theta) \sin^{2b}(\theta) \sin(2\theta)
			\,{\mathrm{d}}\theta 
			&= 
			\int_{0}^{1} 2 (1 - u^{2})^{a} u^{2b + 1} \,{\mathrm{d}}u \\
			&= \int_{0}^{1} (1 - u^{2})^{a} (u^{2})^{b} (2u \,{\mathrm{d}}u) \\
			&= \int_{0}^{1} (1 - t)^{a} t^{b} \,{\mathrm{d}}t.
		\end{align*}

		The desired identity now follows from \Cref{identity:beta-integral}.
	\end{proof}

	\begin{identity} \label{thm:integrating-over-SU2-appendix}
		For nonnegative integers $a, b, c, d$, we have
		\begin{equation*} 
			\int_{\SU_{2}(\mathbb{C})} u_{11}^{a} u_{12}^{b} u_{21}^{c} u_{22}^{d} = 
			\begin{cases}
				(-1)^{b} \dfrac{a! b!}{(a + b + 1)!} & \text{if $a = d$ and $b = c$}, \\[3pt]
				0 & \text{else}.
			\end{cases}
		\end{equation*}
	\end{identity}
	\begin{proof}
		We use the formula for the Haar measure on $\SU_{2}(\mathbb{C})$ from~\Cite[Proposition 7.4.1]{FarautAnalysis}.
		Given a smooth function $f \colon \SU_{2}(\mathbb{C}) \to \mathbb{C}$, we have
		\begin{align*} 
			\int_{\SU_{2}(\mathbb{C})} f 
			&= 
			\frac{1}{2 \pi^{2}} 
			\int_{0}^{\pi/2} 
			\int_{0}^{\pi} 
			\int_{-\pi}^{\pi} 
			f\left(
			\begin{bmatrix}
				e^{\iota \psi} & \\
				& e^{-\iota \psi} \\
			\end{bmatrix}
			\begin{bmatrix}
				\cos(\theta) & \sin(\theta) \\
				-\sin(\theta) & \cos(\theta) \\
			\end{bmatrix}
			\begin{bmatrix}
				e^{\iota \varphi} & \\
				& e^{-\iota \varphi} \\
			\end{bmatrix}
			\right)
			\sin(2 \theta)
			\,{\mathrm{d}}\psi
			\,{\mathrm{d}}\varphi
			\,{\mathrm{d}}\theta \\[8pt]
			&= 
			\frac{1}{2 \pi^{2}} 
			\int_{0}^{\pi/2} 
			\int_{0}^{\pi} 
			\int_{-\pi}^{\pi} 
			f\left(
			\begin{bmatrix}
				e^{\iota (\psi + \varphi)} \cos(\theta) & e^{\iota (\psi - \varphi)} \sin(\theta) \\
				-e^{\iota (-\psi + \varphi)}\sin(\theta) & e^{\iota (-\psi - \varphi)} \cos(\theta) \\
			\end{bmatrix}
			\right)
			\sin(2 \theta)
			\,{\mathrm{d}}\psi
			\,{\mathrm{d}}\varphi
			\,{\mathrm{d}}\theta.
		\end{align*}
		Rewriting in terms of the above coordinates, we get
		\begin{equation*} 
			u_{11}^{a} u_{12}^{b} u_{21}^{c} u_{22}^{d} 
			= 
			(-1)^{c}
			\exp(\iota \psi(a + b - c - d)) 
			\exp(\iota \varphi(a - b + c - d)) 
			\cos^{a + d}(\theta)
			\sin^{b + c}(\theta).
		\end{equation*}
		We integrate using Fubini's theorem to obtain
		\begin{align*} 
			& 2\pi^{2} \int_{\SU_{2}(\mathbb{C})} u_{11}^{a} u_{12}^{b} u_{21}^{c} u_{22}^{d} \\
			&= (-1)^{c} 
			\cdot
			\int_{-\pi}^{\pi} \exp(\iota \psi(a + b - c - d)) \,{\mathrm{d}}\psi
			\cdot
			\int_{0}^{\pi} \exp(\iota \varphi(a - b + c - d)) \,{\mathrm{d}}\varphi
			\cdot 
			\int_{0}^{\pi/2} \cos^{a + d}(\theta) \sin^{b + c}(\theta) \sin(2\theta) \,{\mathrm{d}}\theta.
		\end{align*}
		For the first integral to be nonzero, 
		we must have $a + b - c - d = 0$. 
		This implies that $a - b + c - d$ is even and hence must be zero if the second integral is to be nonzero. 
		Solving these two equations simultaneously gives us $a = d$ and $b = c$. 
		Assume now that these two equations hold. 
		We then have
		\begin{align*} 
			& 2\pi^{2} \int_{\SU_{2}(\mathbb{C})} u_{11}^{a} u_{12}^{b} u_{21}^{c} u_{22}^{d} \\
			&= (-1)^{b} 
			\cdot
			\int_{-\pi}^{\pi} 1 \,{\mathrm{d}}\psi
			\cdot
			\int_{0}^{\pi} 1 \,{\mathrm{d}}\varphi
			\cdot 
			\int_{0}^{\pi/2} \cos^{2a}(\theta) \sin^{2b}(\theta) \sin(2\theta) \,{\mathrm{d}}\theta \\
			&= (-1)^{b} (2 \pi^{2}) \cdot \frac{a!b!}{(a + b + 1)!},
		\end{align*}
		where the last equality follows from \Cref{identity:integrate-cos-sin}.
	\end{proof}

\printbibliography
\end{document}